\documentclass[a4paper,11 pt,reqno]{amsart}
\usepackage{a4,amsbsy,amscd,amssymb,amstext,amsmath,latexsym,oldgerm,enumerate,verbatim,graphicx,xy,ytableau}

\makeatletter                                       
\def\l@subsection{\@tocline{2}{0pt}{2.5pc}{2.5pc}{}}
\makeatother                                        

\makeatletter                                                  
\def\chapter{\clearpage\thispagestyle{plain}\global\@topnum\z@ 
\@afterindenttrue \secdef\@chapter\@schapter}
\makeatother

\newtheorem{thmgl} {Theorem}    
\newtheorem{propgl}{Proposition}
\newtheorem{lemgl} {Lemma}

\theoremstyle{definition}

\newtheorem{remgl} {Remark}
\newtheorem{remsgl} [remgl]{Remarks}

\newcommand{\mf}{\mathfrak}
\newcommand{\mc}{\mathcal}
\newcommand{\mb}{\mathbb}

\newcommand{\nts}{\negthinspace}     
\newcommand{\Nts}{\nts\nts}

\newcommand{\un}{\underline}

\newcommand{\sm}{\setminus}         

\newcommand{\ot}{\otimes}           
\newcommand{\la}{\langle}
\newcommand{\ra}{\rangle}

\newcommand{\End}{{\rm End}}

\newcommand{\Sym}{{\rm Sym}} 

\newcommand{\tr}{{\rm tr}}
\newcommand{\id}{{\rm id}}

\newcommand{\g}{\mf{g}}

\let\ttie\t
\newcommand{\tie}[1]{{\let\t\ttie \ttie#1}}
\renewcommand{\t}{\mf{t}}  

\newcommand{\gl}{\mf{gl}}
\newcommand{\spl}{\mf{sl}}

\newcommand{\GL}{{\rm GL}}

\newcommand{\ve}{\varepsilon}

\makeatletter
\def\vcdots{\vbox{\baselineskip4\p@ \lineskiplimit\z@
\kern3\p@\hbox{.}\hbox{.}\hbox{.}\Nts\nts\kern3\p@}}
\makeatother

\xyoption{all}

\begin{document}

\title{Highest weight vectors for the adjoint action of $\GL_n$ on polynomials}

\begin{abstract}
Let $G=\GL_n$ be the general linear group over an algebraically closed field $k$ and let $\g=\gl_n$ be its Lie algebra. Let $U$ be the subgroup of $G$ which consists of the upper unitriangular matrices. Let $k[\g]$ be the algebra of polynomial functions on $\g$ and let $k[\g]^G$ be the algebra of invariants under the conjugation action of $G$. For certain special weights we give explicit bases for the $k[\g]^G$-module $k[\g]^U_\lambda$ of highest weight vectors of weight $\lambda$. For five of these special weights we show that this basis is algebraically independent over $k[\g]^G$ and generates the $k[\g]^G$-algebra $\bigoplus_{r\ge0}k[\g]^U_{r\lambda}$. Finally we formulate a question which asks whether in characteristic zero $k[\g]^G$-module generators of $k[\g]^U_\lambda$ can be obtained by applying one explicit highest weight vector of weight $\lambda$ in the tensor algebra $T(\g)$ to varying tuples of fundamental invariants.
\end{abstract}

\author[R.\ H.\ Tange]{Rudolf Tange}


\maketitle
\markright{\MakeUppercase{Highest weight vectors for the adjoint action of} $\GL_n$}

\section*{Introduction}
Let $\GL_n$ be the general linear group over an algebraically closed field $k$ and let $\gl_n$ be its Lie algebra. In this paper we will be interested in explicit formulas for highest weight vectors in the ring $k[\gl_n]$ of polynomial functions on $\gl_n$ under the conjugation action. It is natural to take into account the fact that the highest weight vectors of a given weight form a module over the invariant algebra $k[\gl_n]^{\GL_n}$. A crude method would be to map the highest weight vectors in the tensor algebra $T(\gl_n)$ (see e.g. \cite{BCHLLS}) into the symmetric algebra $S(\gl_n)$ which is $\GL_n$-equivariantly isomorphic to $k[\gl_n]$. Mostly one will be projecting to zero. For example, in \cite[Sect.~5 Cor.~2]{PrT} it was shown that the lowest degree in $k[\gl_n]$ where the irreducible of highest weight $n\varpi_1$ occurs is $\frac{1}{2}n(n-1)$. But the lowest degree in $T(\gl_n)$ where this irreducible occurs is $n-1$. Our method involves differentiation of the fundamental invariants and applies to any relevant weight, although we can only prove that it provides a $k[\gl_n]^{\GL_n}$-module basis for a special family of weights.

In \cite{Kos} Kostant showed that, for any reductive group $G$ over $\mb C$, the coordinate rings of the fibers of the adjoint quotient are all isomorphic as $G$-modules to the space $H$ of harmonic functions and determined the multiplicities of the irreducibles in $H$. In \cite{Hes} Hesselink obtained a completely general formula for the graded character of $H$ (or the coordinate ring of the nilpotent cone). For more results on multiplicities in the tensor, symmetric and exterior algebra of the Lie algebra we refer the reader to \cite{Han}, \cite{Stem}, \cite{Gup}, \cite{Bry}, \cite{Ree} and \cite{Baz} and the references in there.

The paper is organised as follows. In Section~\ref{s.prelim} we introduce some basic notation and we recall some results from the literature. Section~\ref{s.semiinvs} contains the main results. Theorem~\ref{thm.basicinvs} gives explicit $k[\gl_n]^{\GL_n}$-module bases for the space of highest weight vectors for a family of $2(n-1)-1$ weights. Theorem~\ref{thm.algindinvs} extends this to all the multiples of $5$ of these weights. Theorems~\ref{thm.basicinvs} and \ref{thm.algindinvs} generalise the results in \cite[Sect.~5]{PrT} for the weight $n\varpi_1$. See also \cite[Lemme~3.4]{Dix2} for the case of the universal enveloping algebra of $\spl_n$. In Section~\ref{s.GL_3} we briefly consider the example $\GL_3$. Here one can actually determine $k[\gl_n]^{\GL_n}$-module bases for the space of highest weight vectors for all relevant weights, that is, one can completely determine the algebra $k[\gl_n]^{U_n}$, where $U_n$ consists of the upper unitriangular matrices. In Section~\ref{s.question} we formulate a question which asks whether in characteristic zero $k[\gl_n]^{\GL_n}$-module generators of $k[\gl_n]^U_\lambda$ can be obtained by applying one explicit highest weight vector of weight $\lambda$ in the tensor algebra $T(\gl_n)$ to varying tuples of fundamental invariants.

\section{Preliminaries}\label{s.prelim}
Throughout this paper $k$ is an algebraically closed field and $G=\GL_n$, $n\ge2$, is the general linear group of invertible $n\times n$ matrices. Its natural module is $V=k^n$ and its Lie algebra is $\g=\gl_n\cong V\ot V^*$. The standard basis elements of $V$ are denoted by $e_1,\ldots,e_n$ and the dual basis elements are denoted by $e_1^*,\ldots,e_n^*$. We identify $\g=\gl_n$ with $\End(V)$, the endomorphisms of the vector space $V$. We denote by $E_{ij}$ the matrix which is $1$ on position $(i,j)$ and $0$ elsewhere. Under the isomorphism $\g\cong V\ot V^*$, $E_{ij}$ corresponds to $e_i\ot e_j^*$. The elements of the dual basis of $E_{ij}$ are denoted by $\xi_{ij}$. So the algebra $k[\g]$ of polynomial functions on $\g$ is a polynomial algebra in the $\xi_{ij}$. The group $G$ acts on $\g$ via the adjoint action (conjugation) and therefore also on $k[\g]$. For any group $H$ and any $kH$-module $W$ we denote the space of $H$-fixed vectors in $W$ by $W^H$.

The Borel subgroup of $G$ which consists of the invertible upper triangular matrices is denoted by $B$ and its unipotent radical, which consists of the upper unitriangular matrices, by $U$. We denote by $T$ the maximal torus of $G$ which consist of the invertible diagonal matrices. The character group of $T$ is denoted by $X$ and its standard basis elements are denoted by $\ve_1,\ldots,\ve_n$. Recall that the positive roots relative to $B$ are the roots $\ve_i-\ve_j$, $i<j$, and that $\lambda\in X$ is dominant if and only if $\lambda_1\ge\lambda_2\ge\cdots\ge\lambda_n$. Furthermore, $\lambda\in X$ occurs in the root lattice if and only if its coordinate sum is $0$. The all-zero and the all-one vector in $X$ are denoted by $\un 0$ and $\un 1$ respectively. For $i\in\{1,\ldots,n-1\}$ the $i$-th fundamental weight $\varpi_i\in\mb Q\ot_{\mb Z}X$ is defined by
$$\varpi_i=\sum_{j=1}^i\ve_j-\frac{i}{n}\,\un 1=\frac{1}{n}\Big((n-i)\sum_{j=1}^i\ve_j-i\sum_{j=i+1}^n\ve_j\Big)\,.$$ The $\mb Z$-span of the fundamental weights contains the root lattice. For $\lambda\in X$ and $W$ a $T$-module we denote the weight space $\{x\in W\,|\,t\cdot x=\lambda(t)x\text{\ for all\ }t\in T\}$ by $W_\lambda$. We denote the irreducible $\GL_n(\mb C)$-module of highest weight $\lambda$ by $L_{\mb C}(\lambda)$. The Weyl group of $G$ relative to $T$ is the symmetric group $\Sym_n$ which permutes the coordinates. We denote the longest Weyl group element by $w_0$. We have $w_0(\ve_i)=\ve_{n-i+1}$, put differently, $w_0(\lambda)$ is the reversed tuple of $\lambda$.

For $i\in\{1,\ldots,n\}$ we define $s_i\in k[\g]$ by $s_i(x)=\tr\wedge^i(x)$, where $\wedge^i(x)$ denotes the $i$-th exterior power of $x$. Then the $s_i$ are up to sign the coefficients of the characteristic polynomial. Note that $s_1=\tr$ and $s_n=\det$. Furthermore, the $s_i$ are algebraically independent generators of $k[\g]^G$. See e.g. \cite[Sect.~7]{Jan2}.

The reader who only wants to understand the precise statements of the main results can now continue to Section~\ref{s.semiinvs}, read definitions \eqref{eq.weights} and \eqref{eq.semiinvs} and then Theorems~\ref{thm.basicinvs} and \ref{thm.algindinvs}.

We now state some auxiliary results that will be needed for the proofs of the main results. The result below was mentioned to me by S.~Donkin.

\begin{lemgl}\label{lem.dimension}
$\dim k[\g]^U=\dim B=\frac{1}{2}n(n+1)$.
\end{lemgl}

\begin{proof}
For $m\in\{1,\ldots,n\}$ put $\Delta_m=\det\big((\xi_{ij})_{n-m+1\le i\le n,1\le j\le m}\big)$. Then $\Delta_m\in k[\g]^U$ for all $m\in\{1,\ldots,n\}$ and $k[\g][\Delta_1^{-1},\ldots,\Delta_n^{-1}]=k[Bw_0B]$. It follows that $k[\g]^U[\Delta_1^{-1},\ldots,\Delta_n^{-1}]=k[Bw_0B]^U$ and $\dim k[\g]^U=\dim k[Bw_0B]^U$. Now $k[Bw_0B]^U\cong k[B]$ via the isomorphism that sends $f\in k[B]$ to the function $uw_0b\mapsto f(bu)$.
\end{proof}

We recall the Graded Nakayama Lemma. For its proof we refer to \cite[Ch.~13]{Pas}, Lem.~4, Ex.~3, Lem.~3.
\begin{lemgl}[{\cite[Ch.~13]{Pas}}]\label{lem.gradednakayama}
Let $S=\bigoplus_{i\ge0}S^i$ be a positively graded ring with $S^0$ a field, let $M$ be a graded $S$-module and let  $(x_i)_{i\in I}$ be a family of homogeneous elements of $M$. Put $S^+=\bigoplus_{i>0}S^i$.
\begin{enumerate}[{\rm(i)}]
\item If the images of the $x_i$ in $M/S^+M$ span the vector space $M/S^+M$ over $S^0$, then the $x_i$ generate $M$.
\item If $M$ is projective and the images of the $x_i$ in $M/S^+M$ form an $S^0$-basis of $M/S^+M$, then $(x_i)_{i\in I}$ is an $S$-basis of $M$.
\end{enumerate}
\end{lemgl}

The closed subvariety of $\g$ which consists of the nilpotent matrices is denoted by $\mc N$. Since $\mc N$ is $G$-stable, $G$ acts on the algebra $k[\mc N]$ of regular functions on $\mc N$.
The two results below are actually valid, under some mild assumptions, for arbitrary reductive groups, but we will not need this generality.

\begin{propgl}[{\cite[Thm.~11]{Kos}, \cite[Sect.~7]{Jan2}, \cite[Thm.~2.2]{Don}, \cite[Prop.~1.3b(i)]{Don2}}]\label{prop.free}\
\begin{enumerate}[{\rm(i)}]
\item The vanishing ideal of $\mc N$ in $k[\g]$ is generated by $s_1,\ldots,s_n$ and for each $\lambda$ the restriction $k[\g]^U_\lambda\to k[\mc N]^U_\lambda$ is surjective and has kernel $(k[\g]^G)^+k[\g]^U_\lambda$.
\item We have $k[\g]^U_\lambda\ne0$ if and only if $\lambda$ is dominant and lies in the root lattice.
\item Assume $\lambda$ is dominant and lies in the root lattice. Then $\dim k[\mc N]^U_\lambda=\dim L_{\mb C}(\lambda)_{\un{0}}$ and $k[\g]^U_\lambda$ is a free $k[\g]^G$-module of rank $\dim L_{\mb C}(\lambda)_{\un{0}}$.
\end{enumerate}
\end{propgl}

Note that $\dim L_{\mb C}(\lambda)_{\un{0}}=\dim L_{\mb C}(-w_0(\lambda))_{\un{0}}$, since the nondegenerate pairing between $L_{\mb C}(\lambda)$ and $L_{\mb C}(-w_0(\lambda))=L_{\mb C}(\lambda)^*$ restricts to one between $L_{\mb C}(\lambda)_{\un{0}}$ and $L_{\mb C}(-w_0(\lambda))_{\un{0}}$.

We will call a weight $\lambda\in X$ {\it primitive} if it is nonzero, dominant, occurs in the root lattice and cannot be written as the sum of two such weights. Note that $k[\g]$ is a unique factorisation domain, since it is isomorphic to a polynomial ring.

\begin{lemgl}\label{lem.irreducible}
Let $u\in k[\g]$ be nonzero. Assume that its top degree term does not vanish on $\mc N$ and is a $B$-semi-invariant of primitive weight $\lambda$. Then $u$ is irreducible.
\end{lemgl}

\begin{proof}
If the top degree term of $u$ is irreducible, then so is $u$. So we may assume that $u$ is homogeneous. We now finish with the arguments from part 3 of the proof of \cite[Prop.~3]{PrT}. Let $u=u_1^{m_1}\cdots u_r^{m_r}$ be the factorisation of $u$ into irreducibles. Then the $u_i$ are homogeneous. By a standard argument using the uniqueness of the prime factorisation and the connectedness of $B$, we get that the $u_i$ are $B$-semi-invariants. Let $\lambda_1,\ldots,\lambda_r$ be their weights. Then these are dominant by \cite[Prop.~II.2.6]{Jan} and we have $\lambda=\sum_{i=1}^rm_i\lambda_i$. So, by the primitivity of $\lambda$, we get that for precisely one $i$, $\lambda_i\ne0$ and for this $i$ we have $m_i=1$. We may assume $i=1$. Then $\lambda_1=\lambda$ and $\lambda_2=\cdots=\lambda_r=0$. So $u_2,\ldots,u_r$ are $B$-invariants and therefore $G$-invariants. Since $u$ is nonzero on $\mc N$, we have by Proposition~\ref{prop.free}(i) that $r=1$.
\end{proof}

\section{The basic semi-invariants}\label{s.semiinvs}
For $t\in\{1,\ldots,n-1\}$ we define the weights
\begin{equation}\label{eq.weights}
\begin{split}
\lambda^t&=\sum_{i=n-t+1}^n(\ve_1-\ve_i)=(t,0,\ldots,0,-1,\ldots,-1)\text{\quad and}\\
\mu^t&=\ \ \sum_{i=1}^t\,(\ve_i-\ve_n)=(1,\ldots,1,0,\ldots,0,-t)\,.
\end{split}
\end{equation}

Note that $\lambda^t$ and $\mu^t$ are dominant and in the root lattice. We have $\lambda^1=\mu^1=\ve_1-\ve_n$ and $\mu^t=-w_0(\lambda^t)$. Furthermore, we have $\lambda^t=t\varpi_1+\varpi_{n-t}$ and $\mu^t=\varpi_t+t\varpi_{n-1}$. A weight $\sum_{i=1}^{n-1}m_i\varpi_i$ occurs in the root lattice if and only if $n|\sum_{i=1}^{n-1}im_i$. From this we easily deduce that $\lambda^t$ and $\mu^t$ are primitive.

All (Young) tableaux that we consider will have entries in $\{1,\ldots,n\}$. Recall that a tableaux is called {\it standard} if the entries in the rows are increasing (i.e. non-decreasing) from left to right and if the entries in the columns are strictly increasing from top to bottom.
\begin{lemgl}\label{lem.mult}
Let $t\in\{1,\ldots,n-1\}$.
\begin{enumerate}[{\rm(i)}]
\item We have $\dim k[\mc N]^U_{\lambda^t}=\dim k[\mc N]^U_{\mu^t}=\binom{n-1}{t}$.
\item Assume $t=1$ or $n\ge3$ and $t\in\{1,n-2,n-1\}$, let $r\ge0$ be an integer and put $s=\binom{n-1}{t}$. Then $\dim k[\mc N]^U_{r\lambda^t}=\dim k[\mc N]^U_{r\mu^t}=\binom{r+s-1}{r}$.
\end{enumerate}
\end{lemgl}

\begin{proof}
(i).\ We only have to consider the case of $\lambda^t$. The given dimension is by Proposition~\ref{prop.free} equal to $\dim L_{\mb C}(\lambda^t)_{\un0}$. Put $\nu:=\lambda^t+\un{1}=(t+1,1,\ldots,1,0,\ldots,0)$, where the number of zeros is $t$. Then $L_{\mb C}(\nu)=\det\ot L_{\mb C}(\lambda^t)$. So it suffices to show that $\dim L_{\mb C}(\nu)_{\un1}=\binom{n-1}{t}$. This dimension is well-known to be equal to the number of standard tableaux of shape $\nu$ and weight $\un1$, that is, each integer in $\{1,\ldots,n\}$ must occur precisely once. The shape $\nu$ is a hook diagram as shown below.
$$
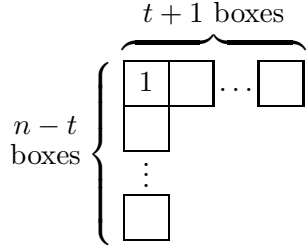
\ytableausetup
{mathmode, boxsize=1.5em}
\begin{xy}
(-9.9,4.0)*=<68pt,28pt>{}*\frm{^\}},
(-9.9,13)*{\text{$t+1$ boxes}},
(-23,-5.6)*=<10pt,68pt>{}*\frm{\{},
(-32,-2)*{\text{$n-t$}},
(-32,-5.6)*{\text{boxes}},
(-9.9,-5.2)*={\begin{ytableau}
1&\ &\none[\dots]&\ \\
\ \\
\none[\vcdots]\\
\
\end{ytableau}}
\end{xy}
$$
\vspace{.01cm}

\noindent Clearly the box in the top left corner must contain $1$ and the tableaux is completely determined by the choices for the other boxes in the first column. So our standard tableaux are in one-one correspondence with the $n-t-1$-subsets of $\{2,\ldots,n\}$.\\
(ii).\ We only have to consider the case of $\lambda^t$. By the same arguments as in (i), it suffices to show that the number of standard tableaux of shape $\nu$ and weight $r\un1$ is $\binom{r+s-1}{r}$, where $\nu:=r\lambda^t+r\un{1}$. So each integer in $\{1,\ldots,n\}$ must occur precisely $r$ times.
First assume $t=1$. Then $s=n-1$ and the shape $\nu$ is a diagram as shown below.
$$
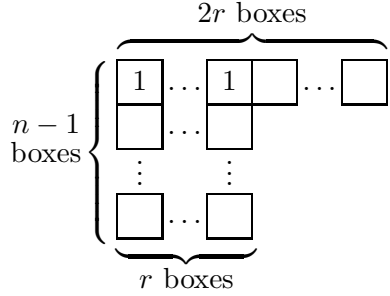
\ytableausetup
{mathmode, boxsize=1.5em}
\begin{xy}
(-9.9,4.0)*=<100pt,28pt>{}*\frm{^\}},
(-9.9,13)*{\text{$2r$ boxes}},
(-28.5,-5.6)*=<10pt,68pt>{}*\frm{\{},
(-37,-2)*{\text{$n-1$}},
(-37,-5.6)*{\text{boxes}},
(-18.6,-14.5)*=<52pt,28pt>{}*\frm{_\}},
(-18.5,-22.3)*{\text{$r$ boxes}},
(-9.9,-5.2)*={\begin{ytableau}
1&\none[\dots]&1&\ &\none[\dots]&\ \\
\ &\none[\dots]&\ \\
\none[\vcdots]&\none&\none[\vcdots]\\
\ &\none[\dots]&\
\end{ytableau}}
\end{xy}
$$
\vspace{.01cm}

\noindent Clearly the first $r$ boxes in the top row must contain $1$. If we ignore the first row, then each column is a strictly increasing subsequence of $\{2,\ldots,n\}$ of length $n-2$. So it is determined by an integer from $\{2,\ldots,n\}$ (the one that does not occur). If we write these in the order of the columns, then the standardness implies that we get an increasing sequence. This sequence is what goes in the final $r$ boxes in the first row and it determines the tableaux completely. The number of such sequences is the same as the number of monomials of degree $r$ in $n-1$ variables, so it equals $\binom{n+r-2}{r}$.

Now assume that $t=n-2$. Then $s=n-1$ and the shape $\nu$ is a diagram as shown below.
$$
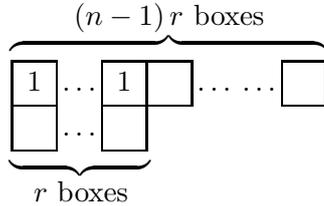
\ytableausetup
{mathmode, boxsize=1.5em}
\begin{xy}
(-10,-2)*=<118pt,28pt>{}*\frm{^\}},
(-9.9,6.5)*{\text{$(n-1)\,r$ boxes}},
(-21.7,-8.5)*=<52pt,28pt>{}*\frm{_\}},
(-21.5,-16.5)*{\text{$r$ boxes}},
(-9.9,-5.2)*={\begin{ytableau}
1&\none[\dots]&1&\ &\none[\dots]&\none[\dots]&\ \\
\ &\none[\dots]&\
\end{ytableau}}
\end{xy}
$$
\vspace{.01cm}

\noindent Again the first $r$ boxes in the top row must contain $1$. Now the diagram is completely determined by the second row which is an increasing subsequence of $\{2,\ldots,n\}$. So again we get $\binom{n+r-2}{r}$ standard tableaux. The case $t=n-1$ is trivial, since the shape $\nu$ is then a single row of length $nr$.
\end{proof}

We now define some basic $B$-semi-invariants in $k[\g]$. For $t\in\{1,\ldots,n-1\}$ and $I\subseteq\{2,\ldots,n\}$ with $|I|=t$ we define

\begin{equation}\label{eq.semiinvs}
\begin{split}
u_{t,I}&:=\det\big((\partial_{1i}s_j)_{n-t+1\le i\le n, j\in I}\big)\text{\quad and}\\
v_{t,I}&:=\det\big((\partial_{in}s_j)_{1\le i\le t, j\in I}\big).
\end{split}
\end{equation}
\vspace{.1cm}

\noindent Here the indices from $I$ are taken in their natural order and $\partial_{ij}$ is the partial derivative $\frac{\partial}{\partial \xi_{ij}}$.
Note that $u_{t,I}$ and $v_{t,I}$ are homogeneous of degree $(\sum_{j\in I}j)-t$.

Define the involution $\varphi$ of the vector space $\g$ by $\varphi(A)=PA^TP$, where $P$ is the permutation matrix corresponding to $w_0$ and $A^T$ denotes the transpose of $A$. Then $\varphi(g\cdot A)=P(g^{-1})^TP\cdot\varphi(A)$, where the dot denotes conjugation action. If we denote the corresponding automorphism of $k[\g]$ also by $\varphi$, then this formula also holds with $A$ replaced by $f\in k[\g]$. So $\varphi(k[\g]^U_\lambda)=k[\g]^U_{-w_0(\lambda)}$. In accordance with this we have $\varphi(u_{t,I})=\pm v_{t,I}$. 

We set up some notation which will give another, more general, way to construct the elements $u_{t,I}$ and $v_{t,I}$. This will make clear why they are $B$-semi-invariants (see the proof of Theorem~\ref{thm.basicinvs}(ii) below).
If $\lambda$ is a partition, then we denote its length by $l(\lambda)$. For $\lambda^+, \lambda^-\in X$ we put $[\lambda^+,\lambda^-]:=\lambda^+-w_0(\lambda^-)$. It is easy to see that for any $\lambda\in X$ dominant there exist unique partitions $\lambda^+$ and $\lambda^-$ with $l(\lambda^+)+l(\lambda^-)\le n$ and $\lambda=[\lambda^+,\lambda^-]$. In the sequel, when $\lambda^+$ and $\lambda^-$ are introduced after $\lambda$, they are supposed to have these properties. Let  $\lambda$ be a partition of $t$. We define the tableau $T_\lambda$ of shape $\lambda$ by $T_\lambda(i,j)=(\sum_{l=1}^{i-1}\lambda_l)+j$. Furthermore we define the subgroup $C_\lambda$ of the symmetric group $\Sym_t$ as the column stabiliser of $T_\lambda$.
Define the element $A_\lambda$ in the group algebra $k\la\Sym_t\ra$ by $A_\lambda=\sum_{\pi\in C_\lambda}{\rm sgn}(\pi)\pi$. Finally, define $e_\lambda\in V^{\ot t}$ and $e^*_\lambda\in {V^*}^{\ot t}$ by
$$
e_\lambda=\bigotimes_{i=1}^{l(\lambda)}e_i^{\ot\lambda_i}\text{\quad and\quad }
e^*_\lambda=\bigotimes_{i=1}^{l(\lambda)}{e^*_{n-i+1}}^{\ot\lambda_i}.
$$
Then, as is well-known (see e.g. \cite{BCHLLS}), $A_\lambda\cdot e_\lambda$ and $A_\lambda\cdot e^*_\lambda$ are highest weight vectors of weight $\lambda$ and $-w_0(\lambda)$ respectively.

Now let $\lambda=[\lambda^+,\lambda^-]$ be dominant and in the root lattice. Then $\lambda^+$ and $\lambda^-$ are partitions of the same number, $t$ say and we define $E_\lambda\in\g^{\ot t}$ as the element corresponding to $A_{\lambda^+}\cdot e_{\lambda^+}\ot A_{\lambda^-}\cdot e^*_{\lambda^-}\in V^{\ot t}\ot {V^*}^{\ot t}$ under the isomorphism $\g^{\ot t}\cong V^{\ot t}\ot {V^*}^{\ot t}$. By the above, $E_\lambda$ is a highest weight vector of weight $\lambda$.

For each $x\in\g$ we can extend the evaluation at $x$, considered as a linear map $\g^*\to k\subseteq k[\g]$, to a derivation of degree $-1$ of the algebra $k[\g]$. Then the evaluation at $E_{ij}$ extends to the derivation $\partial_{ij}$. So we obtain a $G$-equivariant linear map $\g\to\End(k[\g])$ and therefore also a $G$-equivariant linear map

$$\psi_t:\g^{\ot t}\to\End(k[\g]^{\ot t})\,.$$

\vspace{.2cm}

\noindent We denote the $G$-equivariant multiplication map $k[\g]^{\ot t}\to k[\g]$ by $\vartheta$.

\vfill\eject
\begin{thmgl}\label{thm.basicinvs}
\hspace{-.2cm} Let $t\in\{1,\ldots,n-1\}$ \hbox{and let $\lambda^t,\mu^t,u_{t,I},v_{t,I}$\hspace{-.05cm} be given by\hspace{-.05cm} \eqref{eq.weights}\hspace{-.1cm} and\hspace{-.05cm} \eqref{eq.semiinvs}.}
\begin{enumerate}[{\rm(i)}]
\item The $u_{t,I}$, $I\subseteq\{2,\ldots,n\}$ with $|I|=t$, form a basis of the $k[\g]^G$-module $k[\g]^U_{\lambda^t}$. The same holds for the $v_{t,I}$ and $\mu^t$.
\item Any nontrivial $k$-linear combination of the $u_{t,I}$, $I\subseteq\{2,\ldots,n\}$ with $|I|=t$, is an irreducible $B$-semi-invariant of weight $\lambda^t$. The same holds for the $v_{t,I}$ and $\mu^t$.
\end{enumerate}
\end{thmgl}

\begin{proof}
(i).\ Using the involution $\varphi$ we see that we only have to prove the assertion for $\mu^t$ and the $v_{t,I}$. By Proposition~\ref{prop.free} and Lemmas~\ref{lem.gradednakayama} and \ref{lem.mult}(i) it suffices to show that the restrictions of the $v_{t,I}$ to $\mc N$ are linearly independent.
For $\Lambda_1,\Lambda_2\subseteq\{1,\ldots,n\}$ and $A=(a_{ij})_{1\le i,j\le n}\in\g$ set $A_{\Lambda_1,\,\Lambda_2}=(a_{ij})_{(i,j)\in\Lambda_1\times\Lambda_2}$, where the indices are taken in their natural order. Furthermore, put $\mc{X}=(\xi_{ij})_{1\le i,j\le n}$. If $|\Lambda_1|=|\Lambda_2|$, then we have, as in \cite{PrT}, the following basic fact which follows from the Laplace expansion formulae for the determinant:
\begin{equation}\label{eq.diffdet}
\partial_{ij}\big(\det(\mc{X}_{\Lambda_1,\,\Lambda_2})\big)=
\begin{cases}
\pm\det(\mc{X}_{\Lambda_1\sm\{i\},\,\,\Lambda_2\sm\{j\}})\quad\, \mbox{when}\ \, (i,j)\in \Lambda_1\times \Lambda_2,\\
0\qquad\qquad\qquad\qquad\quad\ \ \,\mbox{when}\ \, (i,j)\notin \Lambda_1\times \Lambda_2.
\end{cases}
\end{equation}
For $l\le n$ we have $s_l=\sum_{\Lambda}\det(\mc{X}_{\Lambda,\,\Lambda})$ where the sum ranges over all $l$-subsets $\Lambda$ of $\{1,\ldots,n\}$.

For a sequence $\sigma=(\sigma_1,\ldots,\sigma_s)$ of distinct integers in $\{1,\ldots,n\}$ we define $A_\sigma\in\End(V)$ by $A_\sigma(e_{\sigma_i})=e_{\sigma_{i-1}}$ for $i\in\{2,\ldots,s\}$ and $A_\sigma(e_i)=0$ for $i\notin\{\sigma_2,\ldots,\sigma_s\}$. Then $A_\sigma$ is nilpotent and its restriction to the span of the $e_{\sigma_i}$, $1\le i\le s$, is regular.

If $\Lambda_1,\Lambda_2\subseteq\{1,\ldots,n\}$ with $|\Lambda_1|=|\Lambda_2|>0$ and $\det(\mc{X}_{\Lambda_1,\Lambda_2})(A_\sigma)\ne0$, then
\begin{flalign}\label{eq.evaluation}
\begin{split}
&\bullet\ \Lambda_1\subseteq\{\sigma_1,\ldots,\sigma_{s-1}\}\text{\ and\ }\Lambda_2\subseteq\{\sigma_2,\ldots,\sigma_s\},\\
&\bullet\ \sigma_j\in\Lambda_1 \Rightarrow\sigma_{j+1}\in\Lambda_2\text{\ for all\ }j\in\{1,\ldots,s-1\},\\
&\bullet\ \sigma_j\in\Lambda_2 \Rightarrow\sigma_{j-1}\in\Lambda_1\text{\ for all\ }j\in\{2,\ldots,s\}.
\end{split}&
\end{flalign}

Let $\sigma$ be as above with $\sigma_1=n$. Let $i\in\{1,\ldots,n\}$, let $\Lambda\subseteq\{1,\ldots,n\}$ with $|\Lambda|=l$ and assume that
$\big(\partial_{i n}\det(\mc{X}_{\Lambda,\,\Lambda})\big)(A_\sigma)\ne0$. Then it follows from \eqref{eq.diffdet} and \eqref{eq.evaluation} that $i=\sigma_l$, that $\Lambda=\{\sigma_1,\ldots,\sigma_l\}$ and that $\big(\partial_{i n}\det(\mc{X}_{\Lambda,\,\Lambda})\big)(A_\sigma)=\pm1$.
So for such a $\sigma$ we have $(\partial_{i n}s_l)(A_\sigma)\ne0\Rightarrow l\le s$, $i=\sigma_l$ and $(\partial_{i n}s_l)(A_\sigma)=\pm1$.

So for $\sigma=(\sigma_1,\ldots,\sigma_s)$ and $\tau=(\tau_1,\ldots,\tau_t)$ sequences of distinct integers in $\{1,\ldots,n\}$ and $\pi\in\Sym_t$ with $\sigma_1=n$ and $(\partial_{\pi_1 n}s_{\tau_1})\cdots(\partial_{\pi_t n}s_{\tau_t})(A_\sigma)\ne0$ we have
\begin{enumerate}[{\rm(1)}]
\item $\tau_i\le s$ for all $i\in\{1,\ldots,t\}$,
\item $\sigma\circ\tau=\pi$,
\item $(\partial_{\pi_1 n}s_{\tau_1})\cdots(\partial_{\pi_t n}s_{\tau_t})(A_\sigma)=\pm1$.
\end{enumerate}
Note that (1) implies that $\sigma(\{\tau_1,\ldots,\tau_t\})=\{1,\ldots,t\}$, so the set $\{\tau_1,\ldots,\tau_t\}$ is determined by $\sigma$.

Now we choose for each subset $I=\{i_1>\cdots>i_t\}$ of $\{2,\ldots,n\}$ a sequence $\sigma(I)$ of $i_1\ge t+1$ distinct integers in $\{1,\ldots,n\}$ with $\sigma(I)_1=n$ and $\sigma(I)_{i_j}=j$ for all $j\in\{1,\ldots,t\}$. Then we get for $I,J\subseteq\{2,\ldots,n\}$ with $|I|=|J|=t$ that
$$v_{t,I}(A_{\sigma(J)})=
\begin{cases}
\pm1\text{\ \ if\ } I=J,\\
\,0\text{\ \quad otherwise.}
\end{cases}
$$
So the linear map $f\mapsto f(A_{\sigma(J)})_J:k[\mc N]\to k^{\binom{n-1}{t}}$ sends the family $(v_{t,I}|_{\mc N})_I$ to a basis and
therefore the restrictions of the $v_{t,I}$ to $\mc N$ are linearly independent.\\
(ii).\ Let $I\subseteq\{2,\ldots,n\}$ with $|I|=t$ and write $I=\{i_1<\cdots<i_t\}$. Then it follows immediately from the definitions that
$u_{t,I}=\vartheta\big(\psi_t(F)\cdot s_{i_1}\ot\cdots\ot s_{i_t}\big)$, where
$F=\sum_\pi{\rm sgn}(\pi)E_{1,\pi_{n-t+1}}\ot\cdots\ot E_{1,\pi_n}$, the sum over all permutations $\pi\in\Sym(\{n-t+1,\ldots,n\})$. Now $\lambda_t^+=t\ve_1$ and $\lambda_t^-=\ve_1+\cdots+\ve_t$. So $A_{\lambda_t^+}=\id$, $A_{\lambda_t^-}=\sum_{\pi\in\Sym_t}{\rm sgn}(\pi)\pi$, $e_{\lambda_t^+}=e_1^{\ot t}$, $e_{\lambda_t^-}=e_{n-t+1}^*\ot\cdots\ot e_{n}^*$. It follows that under the isomorphism $\g^{\ot t}\cong V^{\ot t}\ot {V^*}^{\ot t}$, $F$ corresponds to $A_{\lambda_t^+}\cdot e_{\lambda_t^+}\ot A_{\lambda_t^-}\cdot e^*_{\lambda_t^-}$. So $F=E_{\lambda^t}$.
Similarly, we get $v_{t,I}=\vartheta\big(\psi_t(E_{\mu^t})\cdot s_{i_1}\ot\cdots\ot s_{i_t}\big)$. Since the $s_i$ are invariants, this shows that $u_{t,I}$ and $v_{t,I}$ are $B$-semi-invariants of the given weights. Since $\lambda_t$ and $\mu_t$ are primitive, the assertion follows from Lemma~\ref{lem.irreducible} and the linear independence proved in (i).
\end{proof}

\begin{remsgl}
1.\ In \cite[Rem.~26]{Kos} Kostant gave an explicit basis for the isotypic component of the space of harmonics $H$ corresponding to the highest root. So the statement of  Theorem~\ref{thm.basicinvs} in the case of $\lambda^1$ extends to all complex reductive groups.\\
2.\ Assume $k=\mb C$, let $t\le s$ and let $\lambda=[\lambda^+,\lambda^-]$ be dominant and in the root lattice with $\lambda^+$ and $\lambda^-$ partitions of $t$. Then $(\g^{\ot s})^U_\lambda\cong(V^{\ot s}\ot {V^*}^{\ot s})^U_\lambda$ is a simple module for the {\it walled Brauer algebra} ${\mf B}_{s,s}(n)$, see \cite{BCHLLS}. Note that in the definition of the vectors $t_{\tau,\un m, \un n}$ in \cite[Def.~2.4]{BCHLLS} the symmetrisation can be omitted. Above we only considered the case $s=t$, the lowest tensor power of $\g$ which contains $L_{\mb C}(\lambda)$. Then $(\g^{\ot s})^U_\lambda$ is an irreducible $\Sym_t\times\Sym_t$-module and the ideal of ${\mf B}_{t,t}(n)$ spanned by the diagrams with at least one horizontal edge acts as $0$.\\
3.\ Another natural definition of $e_\lambda$ and $e^*_\lambda$ is $e_\lambda=\bigotimes_{i=1}^{l(\lambda')}\otimes_{j=1}^{\lambda_i'}e_j$ and $e^*_\lambda=\bigotimes_{i=1}^{l(\lambda')}\otimes_{j=1}^{\lambda_i'}e^*_{n-j+1}$, where $\lambda'$ denotes the partition of $t$ whose shape is the transpose of that of $\lambda$. In the definition of $A_\lambda$ one then has to replace $T_\lambda$ by its transpose (or $C_\lambda$ by the row stabiliser $R_\lambda$). Then $A_\lambda\cdot e_\lambda$ and $A_\lambda\cdot e^*_\lambda$ are again highest weight vectors of weight $\lambda$ and $-w_0(\lambda)$ and one can define $E_\lambda$ as before. Note that this $E_\lambda$ is $\Sym_t\times\Sym_t$-conjugate to the original one.\\
4.\ Assume $k=\mb C$. Theorem~\ref{thm.basicinvs} answers the so-called {\it first occurrence} question for $k[\g]$ and the weights $\lambda^t$ and $\mu^t$: The lowest degree where $L_{\mb C}(\lambda^t)$ (or $L_{\mb C}(\mu^t)$) occurs in $k[\g]$ is $(\sum_{i=2}^{t+1}i)-t=\frac{1}{2}t(t+1)$.
\end{remsgl}

\begin{thmgl}\label{thm.algindinvs}
Assume $t=1$ or $n\ge3$ and $t\in\{1,n-2,n-1\}$. Then the $u_{t,I}$, $I\subseteq\{2,\ldots,n\}$ with $|I|=t$, are algebraically independent over $k[\g]^G$ and generate the $k[\g]^G$-algebra $\bigoplus_{r\ge0} k[\g]^U_{r\lambda^t}$. Furthermore, the same holds for the $v_{t,I}$ and $\mu^t$.
\end{thmgl}

\begin{proof}
Using the involution $\varphi$ we see that we only have to prove the assertion for $\mu^t$ and the $v_{t,I}$. By Proposition~\ref{prop.free} and Lemmas~\ref{lem.gradednakayama} and \ref{lem.mult}(ii) it suffices to show that the restrictions of the $v_{t,I}$ to $\mc N$ are algebraically independent. If $t=n-1$, then this follows from the fact that $v_{n-1,\{2,\ldots,n\}}|_{\mc N}$ is nonzero by Theorem~\ref{thm.basicinvs}(i) and of degree $>0$. Now we observe the following. If $f_1,\ldots,f_l\in k[\mc N]$, then the morphism $(f_1,\ldots,f_l):\mc N\to k^l$ is dominant if and only if the $f_i$ are algebraically independent and its differential at a point $x\in\mc N$ is surjective if and only if the differentials at $x$ of the $f_i$ are linearly independent. So, by \cite[AG 17.3]{Bo}, it suffices to show that the differentials of the $v_{t,I}|_{\mc N}$ are linearly independent at some smooth point $x\in\mc N$. For $x\in\mc N$ we have that $T_x(\mc N)$ is the intersection of the kernels of the differentials $d_xs_i$ and $x$ is a smooth point of $\mc N$ if and only if the $d_xs_i$ are linearly independent. So it suffices to show that the differentials of the $s_i$ and the $v_{t,I}$ at some nilpotent element $x$ are together linearly independent. We will take $x=A=A_\sigma$, where $\sigma=(n,n-1,\ldots,1)$ and the notation is as in the proof of Theorem~\ref{thm.basicinvs}(i). Put $$\alpha=\big((1,1),\ldots,(1,n),(n,1),\ldots,(n,n-2),(2,1)\big)\,.$$ Let $M$ be the Jacobian matrix of $s_1,\ldots,s_n$ and the $v_{t,I}$ and let $M_\alpha$ be the $(2n-1)$-square submatrix of $M$ consisting of the columns with indices from $\alpha$. We will show that $\det(M_\alpha)(A)=\pm1$. This will prove the required linear independence.

From \eqref{eq.diffdet} and \eqref{eq.evaluation} we deduce easily that $(\partial_{ni}s_j)(A)=0$ and $(\partial_{21}s_j)(A)=0$ for all $i\in\{1,\ldots,n-2\}$ and $j\in\{1,\ldots,n\}$ and that $(\partial_{1i}s_j)(A)=\pm\delta_{ij}$ for all $i,j\in\{1,\ldots,n\}$. So it suffices to show that the matrix $(\partial_{\alpha_i}v_{t,J})(A)_{n+1\le i\le 2n-1, J}$ is diagonal with the diagonal entries equal to $\pm1$, when the subsets $J$ are suitably ordered.

Assume $t=n-2$. For $j\in\{2,\ldots,n\}$ put $w_j=v_{t,\{2,\ldots,n\}\sm\{j\}}$. Put $\tau(j)=(2,\ldots,j-1,j+1,\ldots,n)$. Then we have
\begin{equation}\label{eq.diffw_j}
\partial_{\alpha_m}w_j=\partial_{\alpha_m}\sum\pm(\partial_{\pi_1,n}s_{\tau(j)_1})\cdots(\partial_{\pi_{n-2},n}s_{\tau(j)_{n-2}})\,,
\end{equation}
where the sum is over all $\pi\in\Sym(\{1,\ldots,n-2\})$. We can expand this further by applying the product rule for differentiation. Then each term in \eqref{eq.diffw_j} produces $n-2$ terms, the differentiation $\partial_{\alpha_m}$ being applied to each factor in turn. As in the proof of Theorem~\ref{thm.basicinvs} we have
\begin{equation}\label{eq.diffs_l}
(\partial_{in}s_l)(A)\ne0\Rightarrow (\partial_{in}s_l)(A)=\pm1\text{\ and\ }i=\sigma_l=n-l+1.
\end{equation}

Now assume $j\ge 3$, i.e. $\sigma_j\le n-2$. Then $\sigma_{\tau(j)_1}=\sigma_2=n-1$. Since $\pi$ never takes the value $n-1$, the only term in the expanded form of
\begin{equation}\label{eq.term}
\partial_{\alpha_m}\big((\partial_{\pi_1,n}s_{\tau(j)_1})\cdots(\partial_{\pi_{n-2},n}s_{\tau(j)_{n-2}})\big)
\end{equation}
that can be nonzero at $A$ is
$\big(\partial_{\alpha_m}(\partial_{\pi_1,n}s_2)\big)(\partial_{\pi_2,n}s_{\tau(j)_2})\cdots(\partial_{\pi_{n-2},n}s_{\tau(j)_{n-2}})$.
By \eqref{eq.diffs_l} we must then have $\pi_i=\sigma_{\tau(j)_i}$ for all $i\in\{2,\ldots,n-2\}$ and $\pi_1=\sigma_j$. Finally \eqref{eq.diffdet} and \eqref{eq.evaluation} give us then that $\alpha_m=(n,\sigma_j)$ and that the value of the  term is $\pm 1$.

Now assume that $j=2$. Then $\tau(2)=(3,\ldots,n)$. So for a term in the expanded form of \eqref{eq.term} to be nonzero at $A$ we must, by \eqref{eq.diffs_l}, have $\pi_i=\sigma_{\tau(2)_i}$ for all but one and therefore for all $i\in\{1,\ldots,n-2\}$. So $\pi=(n-2,\ldots,1)$. Now we check that $\big(\partial_{nl}(\partial_{\pi_i,n}s_{\tau(2)_i})\big)(A)=0$ for all $l,i\in\{1,\ldots,n-2\}$ by considering a term $\det(A_{\Lambda\sm\{\pi_i,n\},\,\Lambda\sm\{l,n\}})$ for $\Lambda\subseteq\{1,\ldots,n\}$ with $|\Lambda|=\tau(2)_i=i+2$. Assume first $1\in\Lambda$. Then $\pi_i=1$, since otherwise the first row of $A_{\Lambda\sm\{\pi_i,n\},\,\Lambda\sm\{l,n\}}$ would be zero. So $i=n-2$ and $\Lambda=\{1,\ldots,n\}$. But then the column of index $n-1$ in $A_{\Lambda\sm\{\pi_i,n\},\,\Lambda\sm\{l,n\}}$ is zero. So $1\notin\Lambda$. The cases $l<\pi_i$ and $l=\pi_i$ are now easily dealt with using \eqref{eq.diffdet} and \eqref{eq.evaluation}. So assume $\pi_i<l$. Then we get, using \eqref{eq.diffdet} and \eqref{eq.evaluation}, $\Lambda=\{\pi_i,\ldots,l,n\}$. Then $i+2=|\Lambda|=l-n+i+3$, so $l=n-1$, which is impossible. Finally we check that $\big(\partial_{2,1}(\partial_{\pi_i,n}s_{i+2})\big)(A)=\pm\delta_{i,n-2}$, by considering a term $\det(A_{\Lambda\sm\{2,\pi_i\},\,\Lambda\sm\{1,n\}})$ for $\Lambda\subseteq\{1,\ldots,n\}$ with $|\Lambda|=i+2$. Since $1\in\Lambda$ we must have $\pi_i=1$, so $i=n-2$ and $\Lambda=\{1,\ldots,n\}$. The value of this term is then $\pm1$.

In conclusion we have shown that, for $m\in\{n+1,\ldots,2n-1\}$ and $j\in\{2,\ldots,n\}$, $(\partial_{\alpha_m}w_j)(A)=\pm\delta_{m-n,w_0(j)}$.

Now assume $t=1$. Then we put $w_j=v_{i,\{j\}}=\partial_{1,n}s_j$ and we show that, for $m\in\{n+1,\ldots,2n-1\}$ and $j\in\{2,\ldots,n\}$, $(\partial_{\alpha_m}w_j)(A)=\pm\delta_{m-n,j-1}$. Since this case is much easier we leave it to the reader.
\end{proof}

\begin{remsgl}
1.\ Assume $k=\mb C$, let $t\in\{1,n-2,n-1\}$ and let $r\ge0$. Then, by Theorem~\ref{thm.algindinvs}, the lowest degree where $L_{\mb C}(r\lambda^t)$ (or $L_{\mb C}(r\mu^t)$) occurs in $k[\g]$ is $r\big((\sum_{i=2}^{t+1}i)-t\big)=\frac{1}{2}rt(t+1)$.\\
2.\ Computer calculations suggest that for $t\notin\{1,n-2,n-1\}$ and $r\ge2$ $\dim k[\mc N]^U_{r\lambda^t}<\binom{r+s-1}{r}$, where $s=\dim k[\mc N]^U_{\lambda^t}$. So for such $t$ one cannot expect the $u_{t,I}$ to be algebraically independent over $k[\g]^G$, but one could still conjecture that they generate the $k[\g]^G$-algebra $\bigoplus_{r\ge0} k[\g]^U_{r\lambda^t}$. Similar remarks apply to $\mu^t$ and the $v_{t,I}$.\\
3.\ With a bit more effort one can show that the matrix $M_\alpha(A)$ from the proof of Theorem~\ref{thm.algindinvs} is diagonal with the diagonal entries equal to $\pm1$.
\end{remsgl}

\section{$\GL_3$}\label{s.GL_3}
In this section we describe the algebra $k[\g]^U$ in the case of $\GL_3$. So throughout this section $n=3$, $G=\GL_3$ and $\g=\gl_3$. We have $\lambda^1=\mu^1=\varpi_1+\varpi_2$, $\lambda^2=3\varpi_1=(2,-1,-1)$ and $\mu^2=3\varpi_2=(1,1,-2)$. Note that a weight $l_1\varpi_1+l_2\varpi_2$ is in the root lattice if and only if $3|(l_1-l_2)$. Put $\mc X=(\xi_{ij})_{1\le i,j\le 3}$. For $i,j\in\{1,2,3\}$ we denote by $\mc X^{(i,j)}$ the matrix $\mc X$ with the $i$-th row and $j$-th column omitted and we denote its determinant by $|\mc X^{(i,j)}|$. We put
\begin{align*}
d_1&=\xi_{21}|\mc X^{(1,3)}|+\xi_{31}|\mc X^{(1,2)}|=-u_{2,\{2,3\}}\text{\ \ and}\\
d_2&=\xi_{31}|\mc X^{(2,3)}|+\xi_{32}|\mc X^{(1,3)}|=v_{2,\{2,3\}}\,.
\end{align*}

\begin{lemgl}\label{lem.multGL_3}
Let $\lambda=l_1\varpi_1+l_2\varpi_2$ be dominant and in the root lattice. Put $a=\min(l_1,l_2)$. Then $\dim L_{\mb C}(\lambda)_{\un{0}}=a+1$.
\end{lemgl}

\begin{proof}
Put $b=\frac{1}{3}(l_1+2l_2)$ and $\nu=\lambda+b{\bf1}=(l_1+l_2,l_2,0)$. Then $L_{\mb C}(\nu)=\det^b\ot L_{\mb C}(\lambda)$. So it suffices to show that there are $a+1$ standard tableaux of shape $\nu$ and weight $b\bf1$. This we leave as an exercise for the reader. One has to distinguish the cases $l_1\ge l_2$ and $l_2\ge l_1$.
\end{proof}

\begin{propgl}\label{prop.GL_3}\
\begin{enumerate}[{\rm(i)}]
\item Let $\lambda=l_1\varpi_1+l_2\varpi_2$ be dominant and in the root lattice and put $a=\min(l_1,l_2)$. Put $d=d_1^{(l_1-l_2)/3}$ if $l_1\ge l_2$ and put $d=d_2^{(l_2-l_1)/3}$ otherwise. Then the elements $d\,\xi_{31}^i|\mc X^{(1,3)}|^{a-i}$, $0\le i\le a$, form a basis of the $k[\g]^G$-module $k[\g]^U_\lambda$.
\item The $k$-algebra $k[\g]^U$ is generated by $s_1,s_2,s_3$, $\xi_{31}$, $|\mc X^{(1,3)}|$, $d_1$ and $d_2$. A defining relation is given by $$d_1d_2-|\mc X^{(1,3)}|^3-\xi_{31}|\mc X^{(1,3)}|^2s_1-\xi_{31}^2|\mc X^{(1,3)}|s_2-\xi_{31}^3s_3=0\,.$$
\end{enumerate}
\end{propgl}

\begin{proof}
(i).\ By Proposition~\ref{prop.free} and Lemmas~\ref{lem.gradednakayama} and \ref{lem.multGL_3} it suffices to show that the given elements are independent on $\mc N$. Since they all have different degrees, it suffices to show they are nonzero on $\mc N$. One easily checks that they are all nonzero on $\Big[\begin{smallmatrix}0&0&0\\1&0&0\\1&1&0\end{smallmatrix}\Big]$.\\
(ii).\ By (i) the 7 given elements generate $k[\g]^U$ and by Lemma~\ref{lem.dimension} $\dim k[\g]^U=6$. A straightforward computation shows that the given equation holds and it is clearly irreducible, e.g. by Gauss's Lemma.
\end{proof}

\begin{remgl}
Note that Proposition~\ref{prop.GL_3} also shows that the $k$-algebra $k[\mc N]^U$ is generated by $\xi_{31}$, $|\mc X^{(1,3)}|$, $d_1$ and $d_2$ with defining relation $d_1d_2-|\mc X^{(1,3)}|^3=0$.
\end{remgl}

\section{The method in general}\label{s.question}
As the reader may have noticed after reading the proof of Theorem~\ref{thm.basicinvs}(ii) our method for producing highest weight vectors applies to any dominant weight in the root lattice. So one may wonder whether we always get $k[\g]^G$-module generators. We formulate this as a question. We assume that $k=\mb C$ and use the notation of Section~\ref{s.semiinvs} before Theorem~\ref{thm.basicinvs}.

\medskip

\noindent{\bf Question.} {\it Let $\lambda=[\lambda^+,\lambda^-]$ be dominant and in the root lattice with $\lambda^+$ and $\lambda^-$ partitions of $t$. Do the elements $\vartheta\big(\psi_t(E_\lambda)\cdot s_{i_1}\ot\cdots\ot s_{i_t}\big)$, $2\le i_1,\ldots,i_t\le n$, generate the $k[\g]^G$-module $k[\g]^U_\lambda$? Equivalently, do their restrictions to $\mc N$ span $k[\mc N]^U_\lambda$?}

\medskip

Note that the only thing that varies here is the tuple $(i_1,\ldots,i_t)$. Note also that we allow repetitions in the arguments $s_{i_j}$. As an example we consider the case $n=4$ and $\lambda=2\varpi_2=(1,1,-1,-1)$, a primitive weight. Then the Hesselink-Peterson formula \cite{Hes} shows that $k[\mc N]^U_\lambda$ has dimension $2$ with a generator in degree $2$ and one in degree $4$. We have
$$\vartheta\big(\psi_t(E_\lambda)\cdot s_{i_1}\ot s_{i_2}\big)=\pm\sum{\rm sgn}(\sigma)\,{\rm sgn}(\tau)\partial_{\sigma_1\tau_3}s_{i_1}\partial_{\sigma_2\tau_4}s_{i_2}\,,$$
where the sum is over all $\sigma\in\Sym(\{1,2\})$ and all $\tau\in\Sym(\{3,4\})$.
It follows that $\vartheta\big(\psi_t(E_\lambda)\cdot s_2\ot s_2\big)=\pm2\det(\mc X_{\{3,4\},\{1,2\}})$, where $\mc X_{\{3,4\},\{1,2\}}$ is defined as in the proof of Theorem~\ref{thm.basicinvs}. Clearly this is nonzero on the nilpotent cone. Note that the choice $(s_2,s_2)$ is the only choice that gives the degree $2$ generator. One can check that $(s_3,s_3)$ and $(s_2,s_4)$ both produce semi-invariants of degree $4$ that are nonzero on $\mc N$. In the case $(s_2,s_4)$ it is nonzero on $\mc N$ in any characteristic.

By Theorem~\ref{thm.basicinvs} the answer to our question is affirmative for the weights $\lambda_t$ and $\mu_t$. The basis elements of the spaces $k[\g]^U_{r\lambda^t}$ and $k[\g]^U_{r\mu^t}$, $r>1$ and $t\in\{1,n-2,n-1\}$, from Theorem~\ref{thm.algindinvs} are not formed in accordance with our question.

One can probably formulate a more complicated question for $k$ of arbitrary characteristic, where one divides the expression $\vartheta\big(\psi_t(E_\lambda)\cdot s_{i_1}\ot\cdots\ot s_{i_t}\big)$  by a suitable integer in case of repeated arguments.

\medskip

\noindent{\it Acknowledgement}. This research was funded by a research grant from The Leverhulme Trust.

\bigskip

{\sc\noindent School of Mathematics,
Trinity College Dublin, Dublin~2, Ireland.
{\it E-mail address : }{\tt tanger@tcd.ie}
}

\end{document}